\documentclass[12pt,oneside,english]{amsart}
\usepackage[T1]{fontenc}
\usepackage[latin9]{inputenc}
\usepackage{geometry}
\usepackage{amstext}
\usepackage{amsthm}
\usepackage{amssymb}

\makeatletter

\subjclass[2010]{14D06, 14E05}

\numberwithin{equation}{section}
\numberwithin{figure}{section}
\theoremstyle{plain}
\newtheorem{thm}{\protect\theoremname}[section]
  \theoremstyle{remark}
  \newtheorem{rem}[thm]{\protect\remarkname}
  \theoremstyle{plain}
  \newtheorem{assumption}[thm]{\protect\assumptionname}
  \theoremstyle{definition}
  \newtheorem{defn}[thm]{\protect\definitionname}
  \theoremstyle{plain}
  \newtheorem{prop}[thm]{\protect\propositionname}
  \theoremstyle{plain}
  \newtheorem{lem}[thm]{\protect\lemmaname}
  \theoremstyle{remark}
  \newtheorem{notation}[thm]{\protect\notationname}

\makeatother

\usepackage{babel}
  \providecommand{\assumptionname}{Assumption}
  \providecommand{\definitionname}{Definition}
  \providecommand{\lemmaname}{Lemma}
  \providecommand{\notationname}{Notation}
  \providecommand{\propositionname}{Proposition}
  \providecommand{\remarkname}{Remark}
\providecommand{\theoremname}{Theorem}

\begin{document}

\title{Motivic Serre invariants modulo the square of $\mathbb{L}-1$}

\author{Takehiko Yasuda}
\begin{abstract}
Motivic Serre invariants defined by Loeser and Sebag are elements
of the Grothendieck ring of varieties modulo $\mathbb{L}-1$. In this
paper, we show that we can lift these invariants to modulo the square
of $\mathbb{L}-1$ after tensoring the Grothendieck ring with $\mathbb{Q}$,
under certain assumptions.
\end{abstract}

\address{Department of Mathematics, Graduate School of Science, Osaka University,
Toyonaka, Osaka 560-0043, Japan}

\email{takehikoyasuda@math.sci.osaka-u.ac.jp}

\thanks{The most part of this work was done during the author's stay at Institut
des Hautes Études Scientifiques. He is grateful for its hospitality
and great environment. He also wishes to thank François Loeser for
inspiring discussion and helpful comments. This work was partly supported
by JSPS KAKENHI Grant Number JP15K17510 and JP16H06337.}

\maketitle
\global\long\def\AA{\mathbb{A}}
\global\long\def\PP{\mathbb{P}}
\global\long\def\NN{\mathbb{N}}
\global\long\def\GG{\mathbb{G}}
\global\long\def\ZZ{\mathbb{Z}}
\global\long\def\QQ{\mathbb{Q}}
\global\long\def\CC{\mathbb{C}}
\global\long\def\FF{\mathbb{F}}
\global\long\def\LL{\mathbb{L}}
\global\long\def\RR{\mathbb{R}}

\global\long\def\bx{\mathbf{x}}
\global\long\def\bf{\mathbf{f}}

\global\long\def\cN{\mathcal{N}}
\global\long\def\cW{\mathcal{W}}
\global\long\def\cY{\mathcal{Y}}
\global\long\def\cM{\mathcal{M}}
\global\long\def\cF{\mathcal{F}}
\global\long\def\cX{\mathcal{X}}
\global\long\def\cE{\mathcal{E}}
\global\long\def\cJ{\mathcal{J}}
\global\long\def\cO{\mathcal{O}}
\global\long\def\cD{\mathcal{D}}

\global\long\def\fs{\mathfrak{s}}
\global\long\def\fp{\mathfrak{p}}
\global\long\def\fm{\mathfrak{m}}

\global\long\def\Spec{\mathrm{Spec}\,}
\global\long\def\ord{\mathrm{ord}\,}

\global\long\def\Var{\mathrm{Var}}
\global\long\def\Gal{\mathrm{Gal}}
\global\long\def\bRep{\mathrm{Rep}}
\global\long\def\Jac{\mathrm{Jac}}
\global\long\def\Gor{\mathrm{Gor}}
\global\long\def\Ker{\mathrm{Ker}}
\global\long\def\Im{\mathrm{Im}}
\global\long\def\Aut{\mathrm{Aut}}
\global\long\def\sht{\mathrm{sht}}
\global\long\def\cZ{\mathcal{Z}}
\global\long\def\st{\mathrm{st}}
\global\long\def\diag{\mathrm{diag}}
\global\long\def\characteristic#1{\mathrm{char}(#1)}
\global\long\def\tors{\mathrm{tors}}
\global\long\def\sing{\mathrm{sing}}
\global\long\def\red{\mathrm{red}}
\global\long\def\pt{\mathrm{pt}}
 \global\long\def\univ{\mathrm{univ}}
\global\long\def\length{\mathrm{length}\,}
\global\long\def\sm{\mathrm{sm}}
\global\long\def\top{\mathrm{top}}
\global\long\def\rank{\mathrm{rank\,}}
\global\long\def\age{\mathrm{age}\,}

\section{Introduction}

Let $K$ be a complete discrete valuation field with a perfect residue
field $k$. For a smooth projective irreducible $K$-variety $X$,
Loeser and Sebag \cite{MR1997948} defined the motivic Serre invariant
$S(X)$. This invariant belongs to the ring $K_{0}(\Var_{k})/(\LL-1)$,
where $K_{0}(\Var_{k})$ is the Grothendieck ring of $k$-varieties
and $\LL:=[\AA_{k}^{1}]$, the class of an affine line in this ring.
Let $K_{0}(\Var_{k})_{\QQ}:=K_{0}(\Var_{k})\otimes_{\ZZ}\QQ$. In
this paper, we construct, under a certain assumption, an invariant
\[
\tilde{S}(X)\in K_{0}(\Var_{k})_{\QQ}/(\LL-1)^{2}
\]
which coincides with $S(X)$ in $K_{0}(\Var_{k})_{\QQ}/(\LL-1)$ . 
\begin{rem}
Loeser and Sebag defined the motivic Serre invariant more generally
for smooth quasi-compact separated rigid $K$-spaces. For the sake
of simplicity, we consider only the case where $X$ is a projective
variety. 
\end{rem}
Let $\cO$ be the valuation ring of $K$. The assumption we will make
is that the desingularization theorem and the weak factorization theorem
hold, their precise statements are as follows: 
\begin{assumption}
\label{assu:assumption-desing-weak-fac}
\begin{enumerate}
\item (Desingularization) There exists a regular projective flat $\cO$-scheme
$\cX$ with the generic fiber $\cX_{K}:=\cX\otimes_{\cO}K=X$ such
that the special fiber $\cX_{k}:=\cX\otimes_{\cO}k$ is a simple normal
crossing divisor in $\cX$. (We call such an $\cX$ a \emph{regular
snc model }of $X$.)
\item (Weak factorization) Let $\cX$ and $\cX'$ be regular snc models
of $X$. Then there exist finitely many regular snc models of $X$,
\[
\cX_{0}=\cX,\,\cX_{1},\dots,\cX_{n}=\cX',
\]
such that for every $i$, either the birational map $\cX_{i}\dasharrow\cX_{i+1}$
is the blowup along a regular center $Z\subset\cX_{i+1,k}$ which
has normal crossings\footnote{That $Z$ \emph{has normal crossings with} $\cX_{i+1,k}$ means that
for every closed point $x\in\cX_{i+1,k}$, there exist a regular system
of parameters $x_{1},\dots,x_{d}\in\cO_{\cX_{i+1},x}$ such that in
an open neighborhood of $x$, the support of the special fiber $\cX_{i+1,k}$
is the zero locus of $\prod_{v\in V}x_{v}$ for some subset $V\subset\{1,\dots,d\}$
and $Z$ is the common zero locus of $x_{w}$, $w\in W$ for some
$W\subset\{1,\dots,d\}$. } with $\cX_{i+1,k}$ or its inverse $\cX_{i+1}\dasharrow\cX_{i}$
has the same description with $\cX_{i+1,k}$ replaced with $\cX_{i,k}$. 
\end{enumerate}
\end{assumption}
When $X$ has dimension one, this assumption holds as is well-known.
Indeed the above desingularization theorem in this case follows from
the desingularization theorem for excellent surfaces by Abhyankar,
Hironaka and Lipman (see \cite{MR0491722}), while the weak factorization
follows from the fact that every proper birational morphism of regular
integral noetherian schemes of dimension two factors into a sequence
of finitely many blowups at closed points. The last fact is well-known
in the case of varieties over an algebraically closed field (for instance,
\cite[V, Cor. 5.4]{MR0463157}) and is valid even in our situation
as proved in \cite[Th. 4.1]{MR0276239} in a more general context.
Assumption \ref{assu:assumption-desing-weak-fac} holds also when
$k$ has characteristic zero. This follows from the recent generalizations
to excellent schemes respectively by Temkin \cite{MR2435647,MR2957701}
and by Abramovich and Temkin \cite{Abramovich:2016rr} of the Hironaka
desingularization theorem and the weak factorization theorem of Abramovich,
Karu, Matsuki and W\l odarczyk \cite{MR1896232}.

Let $\cX$ be a regular snc model of $X$, let $\cX_{\sm}$ be its
$\cO$-smooth locus and let $\cX_{\sm,k}:=\cX_{\sm}\otimes_{\cO}k$.
Then $\cX_{\sm}$ is a weak Neron model of $X$ in the sense of \cite{MR1045822}
and by definition, 
\[
S(X)=[\cX_{\sm,k}]\in K_{0}(\Var_{k})/(\LL-1).
\]
To define our invariant $\tilde{S}(X)$, we also need information
on the non-smooth locus of $\cX$. Regard $\cX_{k}$ as a divisor
and write it as $\cX_{k}=\sum_{i\in I}a_{i}D_{i}$, where $D_{i}$
are the irreducible components of $\cX_{k}$ and $a_{i}$ are the
multiplicities of $D_{i}$ in $\cX$ respectively. For a subset $H\subset I$,
we define 
\[
D_{H}^{\circ}:=\bigcap_{h\in H}D_{h}\setminus\bigcup_{i\in I\setminus H}D_{i}.
\]
When $H=\{i\}$, we abbreviate it to $D_{i}^{\circ}$, and when $H=\{i,j\}$,
to $D_{ij}^{\circ}$. These locally closed subsets give the stratification
\[
\cX_{k}=\bigsqcup_{\emptyset\ne H\subset I}D_{H}^{\circ}
\]
and the stratification 
\[
\cX_{\sm,k}=\bigcup_{i\in I:\,a_{i}=1}D_{i}^{\circ}.
\]
From the second stratification, we see 
\[
S(X)=\sum_{i\in I:\,a_{i}=1}[D_{i}^{\circ}]\in K_{0}(\Var_{k})/(\LL-1).
\]
Loeser and Sebag proved in the paper cited above that this is independent
of the model $\cX$ and depends only on $X$. 
\begin{defn}
For a regular snc model $\cX$ of $X$, we define
\[
\tilde{S}(\cX):=\sum_{i\in I:\,a_{i}=1}[D_{i}^{\circ}]+\sum_{\substack{\{i,j\}\subset I:\\
(a_{i},a_{j})=1
}
}\frac{1}{a_{i}a_{j}}[D_{ij}^{\circ}](1-\LL)
\]
as an element of $K_{0}(\Var_{k})_{\QQ}/(\LL-1)^{2}$. Here $(a,b)$
denotes the greatest common divisor of $a$ and $b$. 
\end{defn}
Obviously, the two invariants $S(X)$ and $\tilde{S}(\cX)$ coincide
when they are sent to $K_{0}(\Var_{k})_{\QQ}/(\LL-1)$ by the natural
maps. 

The following is our main theorem:
\begin{thm}
\label{thm:main}Let $X$ be a smooth projective $K$-variety. Under
Assumption \ref{assu:assumption-desing-weak-fac}, the invariant $\tilde{S}(\cX)$
is independent of the chosen regular snc model $\cX$ and depends
only on $X$. 
\end{thm}
The theorem allows us to think of $\tilde{S}(\cX)$ as an invariant
of $X$ and denote it by $\tilde{S}(X)$, which is what was mentioned
at the beginning of this Introduction.

\section{Preparatory reductions}

We generalize the invariant $\tilde{S}(\cX)$ as follows. Let $\cX$
be a regular flat $\cO$-scheme of finite type such that $\cX_{K}$
is smooth and $\cX_{k}=\bigcup_{i\in I}D_{i}$ is a simple normal
crossing divisor in $\cX$. (We no longer suppose that $\cX$ or $\cX_{K}$
is projective.) For a constructible subset $C\subset\cX_{k}$, we
define 
\[
\tilde{S}(\cX,C):=\sum_{\substack{i\in I:\\
a_{i}=1
}
}[D_{i}^{\circ}\cap C]+\sum_{\substack{\{i,j\}\subset I:\\
(a_{i},a_{j})=1
}
}\frac{1}{a_{i}a_{j}}[D_{ij}^{\circ}\cap C](1-\LL)
\]
as an element of $K_{0}(\Var_{k})_{\QQ}/(\LL-1)^{2}$. 

Let $f\colon\cY\to\cX$ be the blowup along a smooth irreducible center
$Z\subset\cX_{k}$ which has normal crossings with $\cX_{k}$. Then,
$\cY$ is an $\cO$-scheme satisfying the same conditions as $\cX$
does and we can similarly define $\tilde{S}(\cY,C')$ for a constructible
subset $C'\subset\cY_{k}$. 

Theorem \ref{thm:main} follows from:
\begin{prop}
\label{prop:main}Let $\cX$ be as above. For any constructible subset
$C\subset\cX_{k}$, we have
\[
\tilde{S}(\cX,C)=\tilde{S}(\cY,f^{-1}(C)).
\]
\end{prop}
Indeed, Theorem \ref{thm:main} is a direct consequence of this proposition
with $C=\cX_{k}$ and Assumption \ref{assu:assumption-desing-weak-fac}. 

In what follows, we will prove this proposition. First we will reduce
it to the local situation by using:
\begin{lem}
\label{lem:reduce-local}
\begin{enumerate}
\item If $C$ is the disjoint union $\bigsqcup_{s=1}^{l}C_{s}$ of constructible
subsets $C_{s}$, then 
\[
\tilde{S}(\cX,C)=\sum_{s=1}^{l}\tilde{S}(\cX,C_{s}).
\]
\item Let $\cX=\bigcup_{\lambda\in\Lambda}U_{\lambda}$ be an open covering.
Suppose that for every constructible subset $C\subset\cX_{k}$ and
for every $\lambda\in\Lambda$, 
\[
\tilde{S}(\cX,C\cap U_{\lambda})=\tilde{S}(\cY,f^{-1}(C\cap U_{\lambda})).
\]
Then, for every constructible subset $C\subset\cX_{k}$, we have 
\[
\tilde{S}(\cX,C)=\tilde{S}(\cY,f^{-1}(C)).
\]
.
\end{enumerate}
\end{lem}
\begin{proof}
The first assertion is obvious. To show the second one, we first claim
that there exists a stratification $C=\bigsqcup_{s=0}^{n}C_{s}$ with
$C_{s}$ constructible such that each $C_{s}$ is contained in some
$U_{\lambda}$. Indeed we can take $C_{0}$ as $C\cap U_{\lambda}$
such that $C$ and $C_{0}$ have equal dimension, then construct $C_{1}$
applying the same procedure to $C\setminus U_{\lambda}$ and so on. 

By the assumption, for every $s$, $\tilde{S}(\cX,C_{s})=\tilde{S}(\cY,f^{-1}(C_{s}))$.
Now, from the first assertion, we get
\[
\tilde{S}(\cX,C)=\sum_{s}\tilde{S}(\cX,C_{s})=\sum_{s}\tilde{S}(\cY,f^{-1}(C_{s}))=\tilde{S}(\cY,f^{-1}(C)).
\]
\end{proof}
Let $x\in\cX_{k}$ be a closed point and take a local coordinate system
$x_{1},\dots,x_{d}\in\cO_{\cX,x}$. By shrinking $\cX$ if necessary,
we may suppose that $x_{1},\dots,x_{d}$ are global sections of $\cO_{\cX}$
and that the special fiber $\cX_{k}$ is the zero locus of $\prod_{i=1}^{d'}x_{i}$,
$d'\le d$ (thus we identify $I$ with $\{1,\dots,d'\}$) and $Z$
is the common zero locus of $x_{j}$, $j\in J$ for some subset $J\subset\{1,\dots,d\}$.
From the first assertion of the above lemma, since we obviously have
\[
\tilde{S}(\cX,C\setminus Z)=\tilde{S}(\cY,f^{-1}(C\setminus Z)),
\]
we may also assume that 
\begin{equation}
C\subset Z.\label{eq:C sub Z}
\end{equation}
In a few following sections, we will prove Proposition \ref{prop:main}
in this situation, discussing separately in the cases $(\sharp I=)d'=1$,
$d'=2$ and $d'\ge3$. Before that, we prepare some notation and a
lemma. 
\begin{notation}
\label{nota:in-prep-reduction}For $i\in I$, let $D_{i}$ be the
prime divisor of $\cX$ given by $x_{i}=0$ and let $E_{i}\subset\cY_{k}$
be its strict transform. Let $E_{0}\subset\cY_{k}$ be the exceptional
divisor of the blowup $f\colon\cY\to\cX$. We denote $f^{-1}(C)$
by $\tilde{C}$. 
\end{notation}
The multiplicity of $E_{i}$ in $\cY_{k}$ is $a_{i}$ for $i\in I$
and 
\begin{equation}
a_{0}:=\sum_{Z\subset D_{i}}a_{i}\label{eq:a0}
\end{equation}
for $i=0$. We will use the following lemma several times.
\begin{lem}
\label{lem:I - J}For $i\in I\setminus J$, if $C\subset Z\cap D_{i}$,
then we have $\tilde{C}\subset E_{i}$. 
\end{lem}
\begin{proof}
The morphism $\tilde{C}\to C$ is a $\PP^{\sharp J-1}$-bundle. The
divisor $E_{i}$ is the blowup of $D_{i}$ along $Z\cap D_{i}$, which
has codimension $\sharp J$ in $D_{i}$. It follows that $E_{i}\cap\tilde{C}\to C$
is also a $\PP^{\sharp J-1}$-bundle. Hence $\tilde{C}$ and $E_{i}\cap\tilde{C}$
coincide and the lemma follows. 
\end{proof}

\section{The case $d'=1$. }

We now begin the proof of Proposition \ref{prop:main} in the situation
described just before Notation \ref{nota:in-prep-reduction}. In this
section, we consider the case $d'=1$.

Since $Z\subset\cX_{k}$, recalling $I=\{1,\dots,d'\}$, we see that
$1\in J$. Then
\[
\tilde{S}(\cX,C)=\begin{cases}
[C] & (a_{1}=1)\\
0 & (\text{otherwise})
\end{cases}.
\]
From (\ref{eq:a0}), $a_{0}=a_{1}$, and $(a_{0},a_{1})=a_{1}$. Hence,
if $a_{1}\ne1$, then
\[
\tilde{S}(\cY,\tilde{C})=0=\tilde{S}(\cX,C).
\]
If $a_{1}=1$, then recalling that $C\subset Z$, we see that $\tilde{C}\subset E_{0}=f^{-1}(Z)$
and that
\begin{align*}
\tilde{S}(\cY,\tilde{C}) & =[\tilde{C}\setminus E_{1}]+[E_{1}\cap\tilde{C}](1-\LL).
\end{align*}
To compute the right hand side of this equality, we first observe
that $\tilde{C}$ is a $\PP^{\sharp J-1}$-bundle over $C$. The divisor
$E_{1}$ is the blowup of $D_{1}$ along $Z$. Therefore $E_{1}\cap\tilde{C}$
is a $\PP^{\sharp J-2}$-bundle over $C$. Hence
\begin{align*}
\tilde{S}(\cY,\tilde{C}) & =[C]([\PP^{\sharp J-1}]-[\PP^{\sharp J-2}])+[C][\PP^{\sharp J-2}](1-\LL)\\
 & =[C]\left(\LL^{\sharp J-1}+(1+\LL+\cdots+\LL^{\sharp J-2})(1-\LL)\right)\\
 & =[C](\LL^{\sharp J-1}+1-\LL^{\sharp J-1})\\
 & =[C]\\
 & =\tilde{S}(\cX,C).
\end{align*}
We conclude that if $d'=1$, then $\tilde{S}(\cX,C)=\tilde{S}(\cY,\tilde{C})$. 

\section{The case $d'=2$. }

Next we consider the case $d'=2$. We have 
\[
C=(C\cap D_{1}^{\circ})\sqcup(C\cap D_{2}^{\circ})\sqcup(C\cap D_{12}^{\circ}).
\]
From the case $\sharp I=1$ treated in the last section, we have
\[
\tilde{S}(\cX,C\cap D_{i}^{\circ})=\tilde{S}(\cY,f^{-1}(C\cap D_{i}^{\circ}))\quad(i=1,2).
\]
Therefore, from Lemma \ref{lem:reduce-local}, replacing $C$ with
$C\cap D_{12}^{\circ}$, we may suppose that 
\begin{equation}
C\subset D_{12}^{\circ}=D_{1}\cap D_{2}.\label{eq:C sub D12}
\end{equation}
Then we have
\[
\tilde{S}(\cX,C)=\begin{cases}
\frac{1}{a_{1}a_{2}}[C](1-\LL) & ((a_{1},a_{2})=1)\\
0 & (\text{otherwise})
\end{cases}.
\]

We next compute $\tilde{S}(\cY,\tilde{C})$ separately in the case
$Z\subset D_{1}\cap D_{2}$ and in the case $Z\not\subset D_{1}\cap D_{2}$. 

In the former case, we have $a_{0}=a_{1}+a_{2}\ne1$ and 
\[
\tilde{S}(\cY,\tilde{C})=\sum_{\substack{i\in\{1,2\}:\\
(a_{0},a_{i})=1
}
}\frac{1}{a_{0}a_{i}}[\tilde{C}\cap E_{0i}^{\circ}](1-\LL).
\]
If $(a_{1},a_{2})\ne1$, then $(a_{0},a_{1})\ne1$ and $(a_{0},a_{2})\ne1$,
which show $\tilde{S}(\cY,\tilde{C})=0=\tilde{S}(\cX,C)$. If $(a_{1},a_{2})=1$,
then we have $(a_{0},a_{1})=(a_{0},a_{2})=1$, and
\begin{align*}
\tilde{S}(\cY,\tilde{C}) & =\sum_{i=1}^{2}\frac{1}{a_{0}a_{i}}[\tilde{C}\cap E_{0i}^{\circ}](1-\LL).
\end{align*}
Since $E_{1}\cap\tilde{C}=E_{0}\cap E_{1}\cap\tilde{C}\to C$ is a
trivial $\PP^{\sharp J-2}$-bundle and $E_{1}\cap E_{2}\cap\tilde{C}\to C$
is a hyperplane in it, $E_{01}^{\circ}\cap\tilde{C}\to C$ is a trivial
$\AA^{\sharp J-2}$-bundle. (Note that if $\sharp J=2$, then $E_{1}\cap E_{2}=\emptyset$
and $E_{1}\cap\tilde{C}=E_{01}^{\circ}\cap\tilde{C}\to C$ is an isomorphism
and still a trivial $\AA^{\sharp J-2}$-bundle.) Similarly for $E_{02}^{\circ}\cap\tilde{C}\to C$.
Hence
\begin{align*}
\tilde{S}(\cY,\tilde{C}) & =\left(\frac{1}{(a_{1}+a_{2})a_{1}}+\frac{1}{(a_{1}+a_{2})a_{2}}\right)[C]\LL^{\sharp J-2}(1-\LL)\\
 & =\frac{1}{a_{1}a_{2}}[C]\LL^{\sharp J-2}(1-\LL)\\
 & \stackrel{\bigstar}{=}\frac{1}{a_{1}a_{2}}[C](1-\LL)\\
 & =\tilde{S}(\cX,C).
\end{align*}
Here the equality marked with $\bigstar$ follows from
\begin{align*}
\LL(1-\LL) & =(\LL-1)(1-\LL)+1-\LL=1-\LL\mod(\LL-1)^{2}.
\end{align*}

In the case $Z\not\subset D_{1}\cap D_{2}$, we have either $Z\subset D_{1}$
or $Z\subset D_{2}$. Since the two cases are similar, we only discuss
the former case. Since $2\in I\setminus J$, from assumptions (\ref{eq:C sub Z})
and (\ref{eq:C sub D12}) and Lemma \ref{lem:I - J}, we have $\tilde{C}\subset E_{0}\cap E_{2}$.
Since $a_{0}=a_{1}$, $\tilde{C}\to C$ is a $\PP^{\sharp J-1}$-bundle
and $\tilde{C}\cap E_{1}\to C$ is a $\PP^{\sharp J-2}$-bundle, we
have

\begin{align*}
\tilde{S}(\cY,\tilde{C}) & =\frac{1}{a_{0}a_{2}}[\tilde{C}\cap E_{0,2}^{\circ}](1-\LL)\\
 & =\frac{1}{a_{1}a_{2}}[\tilde{C}\setminus E_{1}](1-\LL)\\
 & =\frac{1}{a_{1}a_{2}}[C][\PP^{\sharp J-1}\setminus\PP^{\sharp J-2}](1-\LL)\\
 & =\frac{1}{a_{1}a_{2}}[C]\LL^{\sharp J-1}(1-\LL)\\
 & =\frac{1}{a_{1}a_{2}}[C](1-\LL)\\
 & =\tilde{S}(\cX,C).
\end{align*}
We have completed the proof that $\tilde{S}(\cY,\tilde{C})=\tilde{S}(\cX,C),$
when $d'=2$. 

\section{The case $d'\ge3$. }

As in the last section, by induction on $\sharp I$, we may suppose
that 
\begin{equation}
C\subset\bigcap_{i\in I}D_{i}.\label{eq:C sub DI}
\end{equation}
Then $\tilde{S}(\cX,C)=0$. On the other hand, $\tilde{S}(\cY,\tilde{C})$
is a $\QQ$-linear combination of 
\[
A_{i}:=\left[\tilde{C}\cap E_{0i}^{\circ}\right](1-\LL),\,i\in I,
\]
and 
\[
B:=\delta_{1,a_{0}}\left[\tilde{C}\cap E_{0}^{\circ}\right],
\]
with $\delta_{1,a_{0}}$ being the Kronecker delta. Thus it suffices
to show that $A_{i}=0$, $i\in I$ and that $B=0$. 

We first show that $B=0$. If $\sharp(I\cap J)\ge2$, then 
\[
a_{0}=\sum_{i\in I\cap J}a_{i}>1.
\]
Hence $B=0$. If $\sharp(I\cap J)<2$, then $I\setminus J$ is non-empty.
Assumptions (\ref{eq:C sub Z}) and (\ref{eq:C sub DI}) and Lemma
\ref{lem:I - J} show that $\tilde{C}\cap E_{0}^{\circ}$ is empty,
hence $B=0$. 

Next we show that $A_{i}=0$. If $\sharp(I\setminus J)\ge2$, then
from Lemma \ref{lem:I - J}, for every $i\in I$, there exists $i'\in I\setminus\{i\}$
such that $\tilde{C}\subset E_{i'}$. Hence $\tilde{C}\cap E_{0i}^{\circ}=\emptyset$
and $A_{i}=0$. 

If $\sharp(I\setminus J)=1$, then by the same reasoning as above,
$A_{i}=0$ for $i\in I\cap J$. For $i\in I\setminus J$, 
\[
\tilde{C}\cap E_{0i}^{\circ}=\PP_{C}^{\sharp J-1}\setminus\bigcup_{j\in I\cap J}H_{j},
\]
where $\PP_{C}^{\sharp J-1}$ denotes the trivial $\PP^{\sharp J-1}$-bundle
$\PP^{\sharp J-1}\times C$ over $C$ and $H_{j}$ are coordinate
hyperplanes of $\PP_{C}^{\sharp J-1}$. Since $\sharp(I\cap J)\ge2$,
\[
A_{i}=[C][\GG_{m}^{\sharp(I\cap J)-1}\times\AA^{\sharp J-\sharp(I\cap J)}](1-\LL)=-[C]\LL^{\sharp J-\sharp(I\cap J)}(\LL-1)^{\sharp(I\cap J)}=0\mod(\LL-1)^{2}.
\]

If $\sharp(I\setminus J)=0$, equivalently if $Z\subset D_{i}$ for
every $i\in I$, then for every $i\in I$,
\[
\tilde{C}\cap E_{0i}^{\circ}=\PP_{C}^{\sharp J-2}\setminus\bigcup_{j\in I\setminus\{i\}}H_{j},
\]
where $H_{j}$ are coordinate hyperplanes of $\PP_{C}^{\sharp J-2}$.
We have
\[
A_{i}=[C][\GG_{m}^{\sharp I-2}\times\AA^{\sharp J-\sharp I}](1-\LL)=-[C]\LL^{\sharp J-\sharp I}(\LL-1)^{\sharp I-1}=0\mod(\LL-1)^{2}.
\]
We thus have proved that $\tilde{S}(\cX,C)=\tilde{S}(\cY,\tilde{C})=0$
also when $d'\ge3$, which completes the proofs of Proposition \ref{prop:main}
and Theorem \ref{thm:main}. 

\section{Closing comments}

It is natural to try to refine $\tilde{S}(X)$ further by lifting
it to $K_{0}(\Var_{k})_{\QQ}/(\LL-1)^{n}$ for $n>2$ and by adding
extra terms of the form
\[
c[D_{H}^{\circ}](1-\LL)^{\sharp H-1}
\]
with $c\in\QQ$, $H\subset I$, $\sharp H\ge3$. However the author
did not manage to find such a refinement. 

The original invariant considered by Serre \cite{MR0179170} and denoted
by $i(X)$ was defined for a $K$-analytic manifold when the residue
field $k$ is finite, and lives in $\ZZ/(\sharp k-1)$. There seems
to be no counterpart of $\tilde{S}(X)$ in this context, at least
in a naive way, because $\ZZ\otimes_{\ZZ}\QQ=\QQ$ is a field and
the ideal generated by $(\sharp k-1)^{2}$ in it is the entire field. 

The author has no convincing explanation of the meaning of fractional
coefficients appearing in the definition of $\tilde{S}(X)$. However,
as a possibly related work, we note that also Denef and Loeser \cite{MR1815218}
previously considered motivic invariants with coefficients in $\QQ$.

Nicaise and Sebag \cite[Th. 5.4]{MR2285749} gave a nice interpretation
of the Euler characteristic representation of $S(X)$ in terms of
cohomology of the generic fiber (see also \cite{MR3346173} for another
proof). It would be interesting to look for a similar interpretation
of representations of $\tilde{S}(X)$ or $\tilde{S}(X)$ itself. 

\bibliographystyle{plain}
\bibliography{../mybib}

\end{document}